\documentclass[12pt,dvipsnames]{amsart} 

\usepackage[margin=3cm]{geometry}
\usepackage{color}
\usepackage{tikz-cd}
\usepackage[colorlinks,linkcolor=BrickRed,citecolor=OliveGreen,urlcolor=black!55!blue,hypertexnames=true]{hyperref}
\usepackage[shortlabels]{enumitem}
\usepackage{ amssymb, amsthm, amsmath, bbm }

\newtheorem{theorem}{Theorem}[section]
\newtheorem{corollary}[theorem]{Corollary} 
\newtheorem{lemma}[theorem]{Lemma}
\newtheorem{proposition}[theorem]{Proposition}

\newtheorem{algorithm}[theorem]{Algorithm}

\theoremstyle{definition}
\newtheorem{example}[theorem]{Example}

\definecolor{alizarin}{rgb}{0.82, 0.1, 0.26}

\makeatletter \def\l@subsection{\@tocline{2}{0pt}{1pc}{5pc}{}} \def\l@subsection{\@tocline{2}{0pt}{2pc}{6pc}{}} \makeatother

\DeclareMathOperator{\rk}{rk}

\DeclareMathOperator{\Mat}{Mat}
\DeclareMathOperator{\SL}{SL}
\DeclareMathOperator{\GL}{GL}
\DeclareMathOperator{\UU}{U}
\DeclareMathOperator{\OO}{O}
\DeclareMathOperator{\Hom}{Hom}
\DeclareMathOperator{\End}{End}
\DeclareMathOperator{\rad}{rad}
\DeclareMathOperator{\add}{add}
\newcommand{\tra}{{\rm t}}
\newcommand{\C} {{\mathbb C}}
\newcommand{\Z}{{\mathbb Z}}
\newcommand{\N}{{\mathbb N}}
\newcommand{\kk}{{\mathbbm k}}

\newcommand{\Langle}{\mathop{<}\!}
\newcommand{\Rangle}{\!\mathop{>}}
\newcommand{\px}{\kk\Langle x_1,\dots,x_m\Rangle}

\makeatletter
\@namedef{subjclassname@2020}{
	\textup{2020} Mathematics Subject Classification}
\makeatother

\title[Ranks of linear matrix pencils separate similarity orbits]{
Ranks of linear matrix pencils separate \\[.1mm] simultaneous similarity orbits}

\author{Harm Derksen}
\address{Department of Mathematics,
Northeastern University, Boston, MA, USA}
\email{ha.derksen@northeastern.edu}
\thanks{HD was supported by the National Science Foundation grants IIS-1837985 and DMS-2001460.}

\author{Igor Klep}
\address{
Faculty of Mathematics and Physics, 
University of Ljubljana, 
Slovenia}
\email{igor.klep@fmf.uni-lj.si}
\thanks{IK was supported by the Slovenian Research Agency grants J1-2453, N1-0217, J1-3004 and P1-0222.}

\author{Visu Makam}
\address{Radix Trading Europe B. V., Amsterdam, Netherlands}
\email{visu@umich.edu}
\thanks{VM was supported by the University of Melbourne and the National Science Foundation grant CCF-1900460.}

\author{Jurij Vol\v{c}i\v{c}}
\address{
Department of Mathematics, 
Drexel University, 
Pennsylvania, USA}
\email{jurij.volcic@drexel.edu}
\thanks{JV was supported by the National Science Foundation grant DMS-1954709 and the Slovenian Research Agency grant J1-3004.}

\date{\today}

\begin{document}

\begin{abstract}
This paper solves the two-sided version and provides a counterexample to the
general version of the 2003 conjecture by Hadwin and Larson.
Consider evaluations of linear matrix pencils
$L=T_0+x_1T_1+\cdots+x_mT_m$
on matrix tuples as
$L(X_1,\dots,X_m)=I\otimes T_0+X_1\otimes T_1+\cdots+X_m\otimes T_m$.
It is shown that ranks of linear matrix pencils constitute a collection of separating invariants for simultaneous similarity of matrix tuples.
That is, $m$-tuples $A$ and $B$ of $n\times n$ matrices are simultaneously similar if and only if $\rk L(A)=\rk L(B)$ for all linear matrix pencils $L$ of size $mn$.
Variants of this property are also established for symplectic, orthogonal, unitary similarity,
and for the left-right action of general linear groups.
Furthermore, a polynomial time algorithm for orbit equivalence of matrix tuples under the left-right action of special linear groups is deduced.
\end{abstract}

\maketitle

{
\vspace{-2mm}
\linespread{0.95}\small
\tableofcontents
}

\section{Introduction}

Two tuples of $n \times n$ matrices $A = (A_1,\dots,A_m)$ and $B = (B_1,\dots,B_m)$ over a field are {\bf (simultaneously) similar} or conjugate if there exists $P\in \GL_n$ such that $B_i=PA_iP^{-1}$ for $i=1,\dots,m$.
The classification of matrix tuples up to similarity has been deemed a ``hopeless problem'' \cite{LB}.
Nevertheless, the study of simultaneous similarity and related group actions on matrix tuples 
is crucial in multiple areas of mathematics,
ranging from 
operator theory \cite{Fri,DKS},
invariant and representation theory \cite{Dro,Pro} and
algebraic geometry \cite{EH88,LBR} to 
algebraic statistics \cite{AKRS21, DM21} and 
computational complexity \cite{GGOW16, DM17, IQS}. 
As one would expect, this allows for many perspectives in studying matrix tuples and the transfer of ideas across disciplines can be especially fruitful. This paper embodies this spirit -- we leverage results in representation theory to obtain significant results in operator theory and computational complexity. 
Notably, we settle the Hadwin--Larson conjecture \cite{HL} from operator theory, and deduce a polynomial time algorithm for the orbit equivalence of the left-right action which is of interest to complexity theorists, invariant theorists, and algebraic statisticians alike.

A prominent facet of simultaneous similarity 
is finding a (natural) collection of separating invariants.
Note that continuous invariants cannot separate similarity orbits (see e.g. \cite{Pro}).\footnote{As a comparison: it is well-known \cite{Pro} that traces of products of matrices and their complex conjugates form a collection of separating invariants for simultaneous unitary similarity.} 
If an orbit is not closed, any continuous invariant function is forced to take the same value on the entire closure of the orbit, so it is unable to separate orbits whose closures intersect. Indeed, a seminal result of Mumford \cite[Theorem 1.1]{GIT} is that continuous (or even polynomial) invariants 
capture orbit closure intersection: 
the orbit closures of two tuples $A$ and $B$ do not intersect if and only if there is a polynomial invariant $p$ that separates them, i.e., $p(A) \neq p(B)$. A related question is that of the orbit closure inclusion: when is $A$ contained in the closure of the similarity orbit of $B$? 
It is well-known that $A$ and $B$ are similar if and only if $A$ is in the orbit closure of $B$ and $B$ is in the orbit closure of $A$.
Surprisingly, these orbit problems, i.e., orbit equivalence, orbit closure intersection, and orbit closure inclusion have deep connections to central problems in complexity theory, which was unearthed by Mulmuley and Sohoni's Geometric Complexity Theory (GCT) program \cite{MS01, Mul17}. In particular, the VP vs VNP conjecture (an algebraic analog of the celebrated P vs NP conjecture) can be reformulated as the permanent vs determinant problem, the main problem for the GCT approach and manifestly an orbit closure inclusion problem. 

In 1985, Curto and Herrero conjectured \cite[Conjecture 8.14]{CH} that $A$ lies in the closure of the similarity orbit of $B$ if and only if $\rk f(A)\le \rk f(B)$ for every noncommutative polynomial $f$ in $m$ variables. 
Hadwin and Larson in 2003 gave a counterexample \cite[Example 5]{HL} to the (even weaker) two-sided Curto--Herrero conjecture: they presented matrix tuples $A$ and $B$ that are not similar but $\rk f(A)=\rk f(B)$ for every noncommutative polynomial $f$. 
Furthermore, they proposed an ameliorated conjecture \cite[Conjecture 2]{HL}: $A$ lies in the closure of the similarity orbit of $B$ if and only if $\rk F(A)\le \rk F(B)$ for every \emph{matrix noncommutative polynomial} $F$ (i.e., a matrix of noncommutative polynomials).

In this paper we prove the two-sided version of the Hadwin--Larson conjecture, and provide a counterexample to its general version. 
Moreover, we show that only affine linear matrix noncommutative polynomials $F$,
called linear matrix pencils,
of certain size are required for testing rank equality in the two-sided version of the  
conjecture.
\begin{theorem}\label{t:main}
The following are equivalent for $A,B\in \Mat_{n}^m$:
\begin{enumerate}[\rm(i)]
	\item $A$ and $B$ are similar;
	\item for every $T=(T_0,\dots,T_m)\in \Mat_{mn}^{m+1}$,
\begin{equation}\label{e:rank}
\rk\left(I\otimes T_0 + A_1\otimes T_1+ \cdots A_m\otimes T_m\right) =
	\rk\left(I\otimes T_0 + B_1\otimes T_1+ \cdots B_m\otimes T_m\right).
\end{equation}
\end{enumerate}
\end{theorem}
In other words, ranks of linear matrix pencils evaluated at matrix tuples constitute a collection of separating invariants for simultaneous similarity. 
Theorem \ref{t:main} (or rather Theorem \ref{t:leftright} below addressing the left-right multiplication by invertible matrices) also classifies completely rank-preserving maps \cite{Mol,CH1,CH2,HHY}.
This aspect fits under the broader consideration of linear maps preserving various nonlinear properties, such as (complete) positivity.
Furthermore, ranks of linear matrix pencils play an important role in free real algebraic geometry; for example, pencils with same singularity sets are described by noncommutative Nullstellens\"atze \cite{KV,HKV0,HKV}, and low-rank values of a hermitian pencils correspond to extreme points of free spectrahedra \cite{EH}.
Ranks of matrix noncommutative polynomials also pertain to distributions of noncommutative rational functions in free probability \cite{ACSY}.

The proof of Theorem \ref{t:main} is given in Section \ref{sec3}. 
It relies on representation theory of finite-dimensional algebras \cite{Aus82,FNS} and matricization of homomorphisms between finite-dimensional modules.
Section \ref{sec4} gives an analog of Theorem \ref{t:main} for symplectic and orthogonal similarity over an algebraically closed field, and strengthens Theorem \ref{t:main} for unitary and orthogonal similarity over a real closed field. 
In Section \ref{sec5} we first derive a rank condition compatible with the left-right action of general linear groups on matrix tuples (Theorem \ref{t:leftright});
then we present a reduction of the orbit equivalence under the left-right action of special linear groups 
to that of general linear groups (Corollary \ref{c:lam}).
Section \ref{sec6} shows that the general Hadwin--Larson conjecture fails; 
the detailed counterexample is based on an example due to Jon F. Carlson arising from degenerations of modules \cite{Rie,Bon1,Zwa,Sma}.
Finally, algorithmic aspects of our results are collected in Section \ref{sec7}; in particular, we give a polynomial time algorithm for $\SL_p\times \SL_q$ equivalence of matrix tuples (Algorithm \ref{a:sl}).

\subsection*{Acknowledgment}
The authors thank the American Institute of Mathematics for hosting the workshop \emph{Noncommutative inequalities} in June 2021 where this work was initiated.

\section{Preliminaries}\label{sec2}

Throughout the paper let $\kk$ be the underlying field of scalars (without any additional assumptions unless stated otherwise). By $\Mat_{p,q}$ we denote the space of $p\times q$ matrices over $\kk$; for square matrices we write $\Mat_p=\Mat_{p,p}$. 
Given $X\in\Mat_{p,q}^m$ and $P\in\Mat_p$, $Q\in\Mat_q$ we write $PXQ=(PX_1Q,\dots,PX_mQ)$. 
For $i=1,\dots,m$ let $e_i$ denote the column vector with $m$ coordinates that has a $1$ in the $i$\textsuperscript{th} entry and $0$s elsewhere.

Let $\px$ be the free algebra of noncommutative polynomials over $\kk$ in the letters $x_1,\dots,x_m$. While the Hadwin--Larson conjecture \cite[Conjecture 2]{HL} concerns ranks of evaluations of arbitrary matrices over the free algebra, the following proposition shows that it suffices to consider only affine linear matrices over the free algebra.

\begin{proposition}\label{p:matpoly}
For every $F\in \Mat_{d}\otimes\px$ there exists $T=(T_0,\dots,T_m)\in \Mat_{d'}^{m+1}$ such that
\begin{equation}\label{e:lin}
\rk F(A)=\rk \left(I\otimes T_0 + A_1\otimes T_1+ \cdots A_m\otimes T_m\right)-(d'-d)n
\end{equation}
for all $A\in\Mat_{n}^m$ and $n\in\N$.
\end{proposition}

\begin{proof}
Higman's linearization trick \cite[Section 8.5]{Coh} states that
\begin{equation}\label{e:hig}
\begin{pmatrix}
I & f_1 \\ 0 & I
\end{pmatrix}
\begin{pmatrix}
f_0+f_1f_2 & 0 \\ 0 & I
\end{pmatrix}
\begin{pmatrix}
I & 0 \\ -f_2 & I
\end{pmatrix}
=\begin{pmatrix}
f_0 & f_1 \\ -f_2 & I
\end{pmatrix}
\end{equation}
for all matrices $f_0,f_1,f_2$ (over $\px$) of compatible sizes. Applying \eqref{e:hig} recursively we see that
there exists
a linear matrix pencil
$L=T_0+\sum_{i=1}^m T_ix_i \in \Mat_{d'}\otimes\px$ such that
$$P(F\oplus I_{d'-d})Q = L$$
for some invertible $P,Q\in \Mat_{d'}\otimes\px$. Then \eqref{e:lin} clearly holds.
\end{proof}

While well-known to researchers in invariant theory, we state the connection between orbit equivalence and orbit closure inclusion problems for the sake of completeness.

\begin{lemma}\label{l:orbclos}
Let $A,B\in \Mat_{n}^m$. Then $A$ and $B$ are similar if and only if 
$A\in \overline{B^{\GL_n}}$ and $B\in \overline{A^{\GL_n}}$.
\end{lemma}

\begin{proof}
Let $C\in\Mat_{n}^m$. Then the orbit $C^{\GL_n}$ is Zariski open in $\overline{C^{\GL_n}}$ by \cite[Proposition 8.3]{Hum}, and $\overline{C^{\GL_n}}$ is an irreducible variety (since it is the closure of an image of $\GL_n$). Therefore $A^{\GL_n}=B^{\GL_n}$ is equivalent to $\overline{A^{\GL_n}}=\overline{B^{\GL_n}}$.
\end{proof}

\section{Orbit equivalence under similarity}\label{sec3}

First we consider orbit equivalence for the action of $\GL_n$ on $\Mat_n^m$ by similarity.
In this setting, orbits correspond to isomorphism classes of $n$-dimensional modules over a free algebra.
At the heart of our reasoning is the following theorem of Auslander.

\begin{theorem}[{\cite[Proposition 1.5]{Aus82}}] \label{thm:aus}
Let $\Lambda$ be a finite-dimensional $\kk$-algebra, and let $M$ and $N$ be finite-dimensional $\Lambda$-modules. Then $M\cong N$ if and only if
\begin{equation}\label{e:dimhom}
\dim\Hom(X,M)=\dim\Hom(X,N)
\end{equation}
for all finite-dimensional $\Lambda$-modules $X$.
\end{theorem}

Here, $\cong$ denotes isomorphism of $\Lambda$-modules, and $\dim$ denotes the dimension of a vector space over $\kk$. 
We shall rely on the following quantitative strengthening of Theorem \ref{thm:aus} established in \cite{FNS}.
Given the setup as in Theorem \ref{thm:aus}, let $L_0 = M\oplus N$ and inductively define $L_{i+1} = \rad(\End_\Lambda L_i)\cdot L_i\subset L_i$. Then $L_{s+1}=\{0\}$ for large enough $s$, and we let $L=\bigoplus_{i=0}^s L_i$.
Let $\add L$ be the smallest subcategory in the category of finitely generated $\Lambda$-modules that contains $L$ and is closed under direct sums and direct summands. 
By \cite[Proposition 5]{FNS}, $M$ and $N$ are isomorphic if and only if \eqref{e:dimhom} holds for all $X\in \add L$.
The construction of $L_i$ is compatible with direct sums \cite[Remark 4]{FNS}; namely, $L_i\cong M_i\oplus N_i$ for some $M_i\subset M$ and $N_i\subset N$. Consequently, every indecomposable direct summand of $L$ is isomorphic to a submodule of $M$ or $N$ by the Krull--Remak--Schmidt theorem \cite[Corollary 19.22]{Lam}.
This leads to the following statement, alluded to in the proof of \cite[Theorem 6]{FNS}.

\begin{proposition}[{\cite{FNS}}] \label{p:fns}
With the setup as in Theorem \ref{thm:aus}, $M\cong N$ if and only if
\begin{equation}\label{e:dimhom1}
\dim\Hom(X,M)=\dim\Hom(X,N)
\end{equation}
for all indecomposable $\Lambda$-submodules $X$ of $M$ or $N$.
\end{proposition}

\subsection{Proof of Theorem \texorpdfstring{\ref{t:main}}{1.1}}

A tuple $C\in\Mat_{n}^m$ gives rise to a $\px$-module $M_C$, which is the vector space $\kk^n$ with $x_j$ acting on it by matrix multiplication with $C_j$. Conversely, each finite-dimensional $\px$-module is given by a matrix tuple in this way. Note that $M_A$ and $M_B$ are isomorphic as $\px$-modules if and only if $A$ and $B$ are in the same orbit under the similarity action of $\GL_n$.

\begin{lemma}\label{l:m2t}
Let $A\in\Mat_{n}^m$ and $C \in \Mat_{p}^m$. Then $\dim \Hom(M_C,M_A)$ equals the dimension of the kernel of the $mpn\times pn$ matrix  
\begin{equation}\label{e:colmtx}
\begin{pmatrix} I_p\otimes A_1-C_1^\tra\otimes I_n \\ 
 \vdots \\
  I_p\otimes A_m-C_m^\tra\otimes I_n\end{pmatrix}
= \begin{pmatrix} -C_1^\tra \\ \vdots \\ -C_m^\tra \end{pmatrix}\otimes I_n 
+ \sum_{i=1}^m (e_i \otimes I_p)\otimes A_i.
\end{equation}
\end{lemma}

\begin{proof}
The space $\Hom(M_C,M_A)$ is precisely the set of matrices $P \in \Mat_{n,p}$ such that $PC_i = A_iP$ for all $i$. In other words, it is the kernel of the map $\Mat_{n,p} \rightarrow \Mat_{n,p}^m$ given by $P \mapsto (A_1P - PC_1, A_2P - PC_2,\dots,A_mP - PC_m)$. Writing this linear map in coordinates gives us the matrix of \eqref{e:colmtx}.
\end{proof}

\begin{proof}[Proof of Theorem \ref{t:main}]
(i)$\Rightarrow$(ii) clearly holds,  so we consider (ii)$\Rightarrow$(i).
Suppose that $A,B\in\Mat_{n}^m$ are not in the same $\GL_n$-orbit. Let $\Lambda\subset\Mat_{2n}$ be the unital algebra generated by $A_1\oplus B_1,\dots,A_m\oplus B_m$. Then we can view $M_A$ and $M_B$ as $\Lambda$-modules in a natural way. 
Since they are not isomorphic,
by Proposition \ref{p:fns} there exists a $\Lambda$-module $X$ of dimension at most $n$ such that $\dim\Hom(X,M_A)\neq\dim\Hom(X,M_B)$. As a $\px$-module, $X\cong M_C$ for some $C\in\Mat_{n}^m$.
By Lemma \ref{l:m2t} we have
$$\rk\Big( \begin{pmatrix} -C_1^\tra \\ \vdots \\ -C_m^\tra \end{pmatrix}\otimes I_n 
+ \sum_{i=1}^m (e_i \otimes I_t)\otimes A_i \Big)
\neq
\rk\Big( \begin{pmatrix} -C_1^\tra \\ \vdots \\ -C_m^\tra \end{pmatrix}\otimes I_n 
+ \sum_{i=1}^m (e_i \otimes I_t)\otimes B_i \Big).$$
Thus $T_0,\dots,T_m\in\Mat_{mn}$ defined as
\begin{equation}\label{e:c2t}
T_0=-\sum_{j=1}^m (e_je_1^\tra)\otimes C_j^\tra
\quad\text{and}\quad
T_i=(e_ie_1^\tra)\otimes I_t \quad \text{for }i=1,\dots,m
\end{equation}
satisfy
\[\rk\left(I\otimes T_0 + A_1\otimes T_1+ \cdots A_m\otimes T_m\right) \neq
\rk\left(I\otimes T_0 + B_1\otimes T_1+ \cdots B_m\otimes T_m\right).
\qedhere\]
\end{proof}

An algorithm for constructing a rank-disparity witness $T$ in presence of a non-similar pair of tuples is given in Section \ref{sec71}.

\subsection{A bound independent of \texorpdfstring{$m$}{m}}

We can also replace the bound $mn$ on the size of matrices in Theorem \ref{t:main}(2) with one that that is independent of $m$ and depends only on $n$. For $C = (C_1,\dots,C_m)$ and $I = \{i_1 < i_2 < \dots < i_k\} \subseteq \{1,\dots,m\}$ 
we define $C_I
=(C_{i_1},C_{i_2},\dots,C_{i_k})$.

\begin{lemma}\label{l:nom}
Suppose $A,B \in \Mat_{n}^m$. Then $A$ and $B$ are similar if and only if $A_I$ and $B_I$ are similar for all $I \subseteq \{1,\dots,m\}$ with $|I| \leq n^2 + 1$.
\end{lemma}

\begin{proof}
Clearly if $A$ and $B$ are similar, then so are $A_I$ and $B_I$ for all $I$. 
Now suppose $A$ and $B$ are not similar. Take a basis $\{A_{i_1}, A_{i_2},\dots,A_{i_k}\}$ of $\mathrm{span} (A_1,\dots,A_m)$. Let $I = \{i_1,i_2,\dots,i_k\}$. Observe that $k \leq n^2$. 
If $A_I$ is not similar to $B_I$, then we are done. 
Otherwise let $P \in \GL_n$ be such that $PA_IP^{-1} = B_I$. Since $A$ is not similar to $B$, we have $PA_{i_{k+1}}P^{-1} \neq B_{i_{k+1}}$ for some $i_{k+1} \notin I$. Let $I' = I \cup \{i_{k+1}\}$. We claim that $A_{I'}$ is not similar to $B_{I'}$. 
Indeed, if it were, then $QA_{I'}Q^{-1} = B_{I'}$ for some $Q \in \GL_n$. Since $A_{i_{k+1}} = \sum_{1 \leq j \leq k} \lambda_j A_{i_j}$ for some $\lambda_j\in\kk$, it follows that
$$B_{i_{k+1}} = QA_{i_{k+1}}Q^{-1} 
= \sum_j \lambda_j QA_{i_j}Q^{-1}
= \sum_j \lambda_jB_j 
= \sum_j \lambda_j PA_{i_j}P^{-1} 
= PA_{i_{k+1}}P^{-1}$$
which is a contradiction. Hence $A_{I'}$ is not similar to $B_{I'}$ and $|I'| \leq k + 1 \leq n^2 + 1$.
\end{proof}

\begin{corollary}
$A,B\in\Mat_n^m$ are similar if and only if \eqref{e:rank} holds for $T\in\Mat_{n^3+n}^{m+1}$.
\end{corollary}

\begin{proof}
Combine Theorem \ref{t:main} and Lemma \ref{l:nom}.
\end{proof}

\section{Orthogonal, symplectic and unitary similarity}\label{sec4}

In this section we derive the analog of Theorem \ref{t:main} for groups preserving bilinear forms. Throughout the section let $\kk$ be either an algebraically closed field of characteristic 0, or a real closed field.
Given an involution $*$ on $\Mat_n$ and $A=(A_1,\dots,A_m)\in\Mat_n^m$ let
$(A,A^*)=(A_1,\dots,A_m,A_1^*,\dots,A_m^*)\in\Mat_n^{2m}$.

\begin{proposition}\label{p:other}
Let $*$ be an involution on $\Mat_n$ and $G$ a subgroup of $\GL_n$ in one of the following setups:
\begin{enumerate}[\rm(a)]
	\item $\kk$ is real closed or algebraically closed of characteristic 0, $*$ is the transpose and $G$ is the orthogonal group;
	\item $n$ is even, $\kk$ is algebraically closed of characteristic 0, $*$ is the symplectic involution and $G$ is the symplectic group;
	\item $\kk$ is the algebraic closure of a real closed field, $*$ is the conjugate transpose and $G$ is the unitary group.
\end{enumerate}
Then $A,B\in \Mat_n^m$ are $G$-similar if and only if
$(A,A^*),(B,B^*)\in \Mat_n^{2m}$ are $\GL_n$-similar.
\end{proposition}

\begin{proof}
If $B=PAP^{-1}$ for $P\in G$, then also $B^*=PA^*P^{-1}$ since $P^*=P^{-1}$. 
Conversely, suppose that $(A,A^*)$ and $(B,B^*)$ are $\GL_n$-similar.
Then for each word $w$ in letters $x_1,\dots,x_m$ and $x_1^*,\dots,x_m^*$,
the matrices $w(A,A^*)$ and $w(B,B^*)$ are similar and thus have the same trace. 
Then $A$ and $B$ are $G$-similar by 
\cite[Theorems 7.1, 15.3 and 16.4]{Pro} in (a),
\cite[Theorems 10.1 and 15.4]{Pro} in (b), and
\cite[Theorems 11.2 and 16.5]{Pro} in (c).
\end{proof}

\begin{corollary}\label{c:other}
Let $*$ and $G$ be as in Proposition \ref{p:other}.
Then $A,B\in \Mat_n^m$ are $G$-similar if and only if
$$
\rk\Big(I\otimes T_0 + \sum_{i=1}^m (A_i\otimes T_i+A_i^*\otimes T_{i+m}) \Big)
=\rk\Big(I\otimes T_0 + \sum_{i=1}^m (B_i\otimes T_i+B_i^*\otimes T_{i+m}) \Big)
$$
for all $T\in\Mat_{2mn}^{2m+1}$.
\end{corollary}

\begin{proof}
Combine Proposition \ref{p:other} and Theorem \ref{t:main}.
\end{proof}

Using tools from real algebraic geometry \cite{BCR}, Corollary \ref{c:other} can be strengthened for unitary involutions. 
Unless stated otherwise, for the rest of the section let 
$\kk$ be the algebraic closure of a real closed field, let $*$ be the conjugate transpose on $\Mat_n$, and $\UU_n\subset \GL_n$ the unitary group.

\begin{lemma}\label{l:irr} 
Let $C\in\Mat_p^m$ be such that the module $M_{(C,C^*)}$ is irreducible. 
For every $K\in\Mat_n^m$ such that $M_{(C,C^*)}$ does not embed into $M_{(K,K^*)}$, there exists $T\in\Mat_{(2m+1)p}^m$ such that
\begin{align*}
\dim\ker\Big(I\otimes I + \sum_{i=1}^m (C_i\otimes T_i+C_i^*\otimes T_i^*) \Big)=1,\\
\dim\ker\Big(I\otimes I + \sum_{i=1}^m (K_i\otimes T_i+K_i^*\otimes T_i^*) \Big)=0.
\end{align*}
\end{lemma}

\begin{proof}
By Lemma \ref{l:m2t}, the dimension of the kernel of
$$
\begin{pmatrix} I\otimes A_1-C_1^\tra\otimes I \\ 
\vdots \\
I\otimes A_m-C_m^\tra\otimes I \\
I\otimes A_1^*-C_1^{*\tra}\otimes I \\ 
\vdots \\
I\otimes A_m^*-C_m^{*\tra}\otimes I
\end{pmatrix}$$
is $1$ if $A=C$ and $0$ if $A=K$. 
Let $R=\sum_i (C_i^{*\tra}C_i^{\tra}+C_i^{\tra}C_i^{*\tra})$ 
(which is invertible by irreducibility);
then the same conclusion holds for the matrix
\begin{align*}
& (R^{-1}\otimes I)\begin{pmatrix} I\otimes A_1-C_1^\tra\otimes I \\ 
\vdots \\
I\otimes A_m^*-C_m^{*\tra}\otimes I
\end{pmatrix}^*
\begin{pmatrix} I\otimes A_1-C_1^\tra\otimes I \\ 
\vdots \\
I\otimes A_m^*-C_m^{*\tra}\otimes I
\end{pmatrix} \\
=\ & I\otimes I
+\sum_i R^{-1} \otimes(A_i^*A_i+A_iA_i^*)
-2\sum_i(R^{-1}C_i^{\tra}\otimes A_i^*+R^{-1}C_i^{*\tra}\otimes A_i).
\end{align*}
Furthermore, a Schur complement argument then implies that the dimension of the kernel of
\begin{equation}\label{e:bigmtx}
\begin{pmatrix}
I\otimes I &  &  & -R^{-1}\otimes A_1^*\\
 & \ddots &  & \vdots\\
 &  & I\otimes I & -R^{-1}\otimes A_m\\
I\otimes A_1 & \cdots & I\otimes A_m^* & I\otimes I
-2\sum_i(R^{-1}C_i^{\tra}\otimes A_i^*+R^{-1}C_i^{*\tra}\otimes A_i)
\end{pmatrix},
\end{equation}
where the missing blocks are zero,
is $1$ if $A=C$ and $0$ if $A=K$.

In the affine space $\Mat_{(2m+1)p}^{2m}$ consider the sets
\begin{alignat*}{3}
\mathcal{X}&=\bigg\{ T\in\Mat_{(2m+1)p}^{2m}&&\colon 
&&\det\Big(I\otimes I+\sum_i (C_i\otimes T_i+C_i^*\otimes T_{i+m})\Big)=0 \bigg\},\\
\mathcal{Y}&=\bigg\{T\in\Mat_{(2m+1)p}^{2m}&&\colon 
&&\dim\ker\Big(I\otimes I+\sum_i (C_i\otimes T_i+C_i^*\otimes T_{i+m})\Big)=1 \\
& && \&\ &&\dim\ker\Big(I\otimes I+\sum_i (K_i\otimes T_i+K_i^*\otimes T_{i+m})\Big)=0
\bigg\},\\
\mathcal{R}&=\bigg\{ T\in\Mat_{(2m+1)p}^{2m}&&\colon 
&&T_{i+m}=T_i^* \text{ for } 1\le i\le m \bigg\}.
\end{alignat*}
Then $\mathcal{Y}$ is a Zariski open subset of the algebraic set $\mathcal{X}$, and $\mathcal{R}$ is the set of real points in $\Mat_{(2m+1)p}^{2m}$ with respect to the real structure $(U,V)\mapsto (V^*,U^*)$ for $(U,V)\in \Mat_{(2m+1)p}^m\times \Mat_{(2m+1)p}^m=\Mat_{(2m+1)p}^{2m}$.
Note that $\mathcal Y\neq\emptyset$ by \eqref{e:bigmtx}.
The determinant of a monic hermitian pencil is a \emph{real zero polynomial} \cite{HV},
meaning it has only real zeros along every line through the origin.
Since $\mathcal{X}$ is therefore the zero set of a real zero polynomial, 
it follows by \cite[Proposition 5.1]{KV} that $\mathcal{X}\cap\mathcal{R}$ is Zariski dense in $\mathcal{X}$. Therefore $\mathcal{Y}\cap\mathcal{R}\neq\emptyset$, which is the required conclusion.
\end{proof}

The next statement shows that for certifying unitary similarity with the rank equality condition \ref{e:rank},
instead of general $(2m+1)$-tuples as in Corollary \ref{c:other}
it suffices to consider only those of a special form $(I,T,T^*)$ for an $m$-tuple $T$.

\begin{theorem}\label{t:real}
The tuples $A,B\in \Mat_n^m$ are $\UU_n$-similar if and only if
$$\rk\Big(I\otimes I + \sum_{i=1}^m (A_i\otimes T_i+A_i^*\otimes T_i^*) \Big) =
\rk\Big(I\otimes I + \sum_{i=1}^m (B_i\otimes T_i+ B_i^*\otimes T_i^*) \Big)$$
for all $T\in\Mat_{(2m+1)n}^{m}$.
\end{theorem}

\begin{proof}
The modules $M_{(A,A^*)}$ and $M_{(B,B^*)}$ are semisimple \cite[Page 90]{Lam}. 
If they are not isomorphic, then there exists an irreducible module $M_{(C,C^*)}$ for $C\in\Mat_p^m$ for $p\le n$ that appears with distinct multiplicities in $M_{(A,A^*)}$ and $M_{(B,B^*)}$. 
Let $M_{(K,K^*)}$ be the direct sum of all irreducible submodules in $M_{(A,A^*)}$ or $M_{(B,B^*)}$ that are not isomorphic to $M_{(C,C^*)}$. Lemma \ref{l:irr} applied to $C$ and $K$ yields the desired matrix tuple $T$.
\end{proof}

Applying Theorem \ref{t:real} to matrix tuples over the underlying real closed field gives the following.

\begin{corollary}\label{c:real2}
Suppose $\kk$ is a real closed field and $\OO_n\subset \GL_n$ is the orthogonal group.
Then $A,B\in \Mat_n^m$ are $\OO_n$-similar if and only if
$$\rk\Big(I\otimes I + \sum_{i=1}^m (A_i\otimes T_i+A_i^\tra\otimes T_i^\tra) \Big) =
\rk\Big(I\otimes I + \sum_{i=1}^m (B_i\otimes T_i+ B_i^\tra\otimes T_i^\tra) \Big)$$
for all $T\in\Mat_{2(2m+1)n}^{m}$.
\end{corollary}

\begin{proof}
Note that $(2m+1)n\times (2m+1)n$ complex matrices $*$-embed into $2(2m+1)n\times 2(2m+1)n$ real matrices, so the statement follows by Theorem \ref{t:real} and Proposition \ref{p:other}.
\end{proof}

Lastly, Lemma \ref{l:irr} also gives an improved matrix size bound, 
linear in $m$ and in $n$, 
for the quantum version \cite[Corollary 5.7]{KV} of the Kippenhahn conjecture \cite[Section 8]{Kip}.

\begin{corollary}\label{c:kip} 
Let $H\in\Mat_n^m$ be an irreducible tuple of hermitian matrices. There is a tuple of hermitian matrices $T\in\Mat_{(m+1)n}^m$ such that $H_1\otimes T_1+\cdots+H_m\otimes T_m$ has a simple nonzero eigenvalue.
\end{corollary}

\section{Orbit equivalence for the left-right action}\label{sec5}

The left-right action of $\GL_p \times \GL_q$ (and its subgroup $\SL_p \times \SL_q$) on matrix tuples by simultaneous left and right multiplication has been of considerable interest in the past few years. Hrube{\v s} and Wigderson \cite{HW} showed that the orbit closure intersection problem (more precisely, the so-called null cone membership problem for the left-right action of $\SL_n \times \SL_n$) captures the problem of non-commutative rational identity testing. Identity testing problems are key to some of the deepest outstanding problems in complexity theory, see \cite{Mul17, KI04}. Polynomial time algorithms in this case were obtained in recent years \cite{GGOW16, IQS, DM17, DM20}.
These algorithms also inspired progress in other subjects like noncommutative geodesic optimization \cite{BFGO+}, algebraic statistics \cite{AKRS21, DM21}, Brascamp-Lieb inequalities \cite{GGOW-BL}, and the Paulsen problem \cite{Paulsen}.

Even amidst this flurry of activity, a polynomial time algorithm for the orbit equivalence problem for the left-right action of $\SL_p \times \SL_q$-action remained elusive. Note that for the left-right action of $\GL_p \times \GL_q$, a polynomial time algorithm for the orbit equivalence problems follows from the results of Brooksbank and Luks \cite{BL08}. In this section, we develop some structural results regarding orbit equivalence that we then use to give polynomial time algorithms in Section~\ref{sec7}.

\subsection{\texorpdfstring{$\GL_p\times\GL_q$}{GLpxGLq} action}

In this section we consider the action of $\GL_p\times\GL_q$ on $\Mat_{p,q}^m$ by simultaneous left and right multiplication.
Let $\Lambda_m$ be the path algebra of the $m$-Kronecker quiver. That is,
$$\Lambda_m=\kk\Langle e,y_1,\dots,y_m\mid e^2=e, ey_j=y_j, y_iy_j=y_je=0\Rangle.$$
Every $C\in \Mat_{p,q}^m$ determines a finite-dimensional $\Lambda_m$-module $N_C$ with dimension vector $(p,q)$ (and $\dim N_C=p+q$), and vice versa \cite[Section 7.1]{DW}. 
Concretely, $e$ acts on $\kk^p\times \kk^q$ as the projection onto the first component, while $y_j$ acts by matrix multiplication with 
$(\begin{smallmatrix}0 & C_j \\ 0 & 0\end{smallmatrix})$.
Modules $N_A,N_B$ for $A,B\in \Mat_{p,q}^m$ are isomorphic if and only if $A,B$ are in the same $\GL_p \times \GL_q$-orbit.

\begin{lemma}\label{l:m2t2}
Let $A\in\Mat_{p,q}^m$ and $C \in \Mat_{r,s}^m$. Then $\dim \Hom(N_A,N_C)$ equals the dimension of the kernel of the $mps\times (qs+pr)$ matrix
$$
\begin{pmatrix}
I_s\otimes A_1 & -C_1^\tra\otimes I_p\\
\vdots &  \vdots \\
I_s\otimes A_m & -C_m^\tra\otimes I_p\\
\end{pmatrix}
= 
\begin{pmatrix}
\sum_{i=1}^m (e_i \otimes I_s)\otimes A_i &
\begin{pmatrix}
-C_1^\tra\\
\vdots \\
-C_m^\tra
\end{pmatrix}\otimes I_p
\end{pmatrix}.$$
\end{lemma}

\begin{proof}
The space $\Hom(N_C,N_A)$ is identified with the set of pairs $(P,Q)\in \Mat_{q,s}\times\Mat_{p,r}$ such that $QC_i=A_iP$ for all $i$. 
As in the proof of Lemma \ref{l:m2t} we hence view $\Hom(N_C,N_A)$ as the kernel of the linear map $(P,Q)\mapsto (A_1P-QC_1,\dots,A_mP-QC_m)$, and the matrix representation of this map gives the desired conclusion.
\end{proof}

\begin{theorem}\label{t:leftright}
The following are equivalent for $A,B\in \Mat_{p,q}^m$:
\begin{enumerate}[\rm(i)]
\item $A$ and $B$ are in the same $\GL_p\times\GL_q$-orbit;
\item for every $T\in \Mat_{mq-1,q}^m$,
$$\rk\left(A_1\otimes T_1+ \cdots +A_m\otimes T_m\right) =
\rk\left(B_1\otimes T_1+ \cdots +B_m\otimes T_m\right);$$
\item for every $T\in \Mat_{p,mp-1}^m$,
$$\rk\left(A_1\otimes T_1+ \cdots +A_m\otimes T_m\right) =
\rk\left(B_1\otimes T_1+ \cdots +B_m\otimes T_m\right).$$
\end{enumerate}
\end{theorem}

\begin{proof}
(i)$\Rightarrow$(ii),(iii) is straightforward. 
We only need to prove (ii)$\Rightarrow$(i) since (iii)$\Rightarrow$(i) then follows from applying (ii)$\Rightarrow$(i) to $A^\tra,B^\tra$.

If $A$ and $B$ are not in the same $\GL_p\times\GL_q$-orbit, then by Proposition \ref{p:fns} there exists $C\in \Mat_{p,q}^m$ such that $\dim \Hom(N_C,N_A)\neq \dim \Hom(N_C,N_B)$. 
Let $Q\in \GL_{mq}$ and $P\in\GL_p$ be such that
$$Q \begin{pmatrix}
-C_1^\tra\\
\vdots \\
-C_m^\tra
\end{pmatrix}P = \begin{pmatrix}
I_r & 0 \\ 0 & 0
\end{pmatrix}$$
for $r\le p$. By Lemma \ref{l:m2t2} we have
$$
\rk \begin{pmatrix}
\sum_i Q(e_i \otimes I_s)\otimes A_i &
\begin{pmatrix}
I_r & 0 \\ 0 & 0
\end{pmatrix}\otimes I_p
\end{pmatrix}
\neq \rk \begin{pmatrix}
\sum_i Q(e_i \otimes I_s)\otimes B_i &
\begin{pmatrix}
I_r & 0 \\ 0 & 0
\end{pmatrix}\otimes I_p
\end{pmatrix}.
$$
Note that this can only happen if $0<r<mq$ (the first inequality holds since $C\neq0$). Let $T_i \in\Mat_{mq-r,q}$ be obtained by removing the first $r$ rows of $Q(e_i \otimes I_q)$. Then
$\rk(\sum_i A_i \otimes T_i) 
\neq \rk(\sum_i B_i \otimes T_i)$.
\end{proof}

\subsection{\texorpdfstring{$\SL_p\times\SL_q$}{SLpxSLq} action}

Throughout this section let $\kk$ be an algebraically closed field.
Orbit membership in $\Mat_{p,q}^m$ under the left-right action of $\SL_p \times \SL_q$ is more subtle than in the case of $\GL_p\times \GL_q$.
If $p=q=n$ and the tuples $A,B\in\Mat_n^m$ are outside the null cone of the $\SL_n \times \SL_n$ action, we can reduce the $\SL_n \times \SL_n$ equivalence to the $\GL_n$ similarity equivalence by using the ideas from \cite{DM20}. On the other hand, if the tuples are non-square or in the null cone, then $\SL_p \times \SL_q$ orbit membership requires a more refined analysis appealing to some results on preprojective algebras for quivers. Corresponding algorithms for checking $\SL_p\times\SL_q$ equivalence are given in Section \ref{sec72}.

\subsubsection{Reduction from \texorpdfstring{$\SL_n \times \SL_n$}{SLnxSLn} to similarity when outside the null cone}
When detecting orbit equivalence of matrix tuples outside the null cone for the $\SL_n\times \SL_n$-action, the rank equality condition of Theorem \ref{t:leftright} can be supplemented with a determinant equality condition.

\begin{proposition}\label{p:notnull}
Suppose $A,B \in\Mat_n^m$ are in the same $\GL_n \times \GL_n$ orbit and not in the null cone. Then $A$ and $B$ are not in the same $\SL_n \times \SL_n$-times orbit if and only if there exists $d \in \{n-1,n\}$ such that for any choice of $T\in \Mat_d^m$ with $\det(\sum_{i=1}^m A_i \otimes T_i) \neq 0$, we have
$\det(\sum_{i=1}^m A_i \otimes T_i) \neq \det(\sum_{i=1}^m B_i \otimes T_i)$.
\end{proposition}

\begin{proof}
$(\Rightarrow)$ Suppose $A$ and $B$ are in the same $\SL_n \times \SL_n$-orbit.
Then clearly
$\det(\sum_i A_i \otimes T_i) = \det(\sum_i B_i \otimes T_i)$ for all choices of $T$.

$(\Leftarrow)$ Observe that $A$ and $\mu A$ are in the same $\SL_n \times \SL_n$ orbit if $\mu$ is an $n$\textsuperscript{th} root of unity because $\mu I \in \SL_n$.
Now suppose $A$ and $B$ are not in the same $\SL_n \times \SL_n$-orbit, but in the same $\GL_n \times \GL_n$ orbit. Thus $\lambda A$ is in the same $\SL_n \times \SL_n$-orbit as $B$ for some $ \lambda \in \C$, where $\lambda$ is not an $n$\textsuperscript{th} root of unity. Therefore $\lambda^{dn} \neq 1$ 
for some $d\in \{n-1,n\}$. 
Take $d \in \{n-1,n\}$ such that $\lambda^{dn} \neq 1$ and choose any $T \in \Mat_{d}^m$ such that $\det(\sum_i A_i \otimes T_i) \neq 0$. Then
\[
\det(\textstyle\sum_i B_i \otimes T_i) = 
\det(\textstyle\sum_i \lambda A_i \otimes T_i) = 
\lambda^{dn} \det(\textstyle\sum_i A_i \otimes T_i) \neq 
\det(\textstyle\sum_i A_i \otimes T_i).
\qedhere
\]
\end{proof}

\subsubsection{The general case} 
The matter of $\SL_p\times\SL_q$ equivalence of two points in $\Mat_{p,q}^m$ splits into two parts: the $\GL_p\times\GL_q$ equivalence in $\Mat_{p,q}^m$ (Theorem \ref{t:leftright}), and the $\SL_p\times\SL_q$ equivalence of $A$ and $\lambda A$ for $A\in\Mat_{p,q}^m$ and $\lambda\in\C$. In this section we analyze the second part.

\begin{lemma}\label{l:ab}
Let $A = (A_1,\dots,A_m)\in \Mat_{p,q}^m$ and suppose that $A_i = (\begin{smallmatrix} P_i & 0 \\ 0 & Q_i \end{smallmatrix})$ for each $i$ where $P_i$ is of size $k \times \ell$ and $Q_i$ of size $(p-k) \times (q-\ell)$. If $p\ell\neq qk$, then $A$ and $\lambda A$ are in the same $\SL_p \times \SL_q$-orbit for every $0 \neq \lambda \in \C$.
\end{lemma}

\begin{proof}
Choose $\mu$ such that $\mu^{p\ell-qk} = \lambda$.
Now let
$D_1 =\mu^{(p-k)q} I_k \oplus \mu^{-kq} I_{p-k}$ and
$D_2 = \mu^{p(\ell-q)} I_\ell \oplus \mu^{p\ell} I_{q-\ell}$.
Then $D_1\in \SL_p$, $D_2\in\SL_q$ and $D_1 A D_2 = \mu^{p\ell-qk} A = \lambda A$.
\end{proof}

\begin{lemma}\label{l:ab2}
Let $A \in\Mat_{p,q}^m$ and consider the corresponding $\Lambda_m$-module $N_A$.
Then the $\GL_p\times\GL_q$-orbit of $A$ contains $(\begin{smallmatrix} P & 0 \\ 0 & Q \end{smallmatrix})$ where $P\in\Mat_{k,\ell}^m$ and $Q\in\Mat_{p-k,q-\ell}^m$ with $p\ell \neq qk$ if and only if $N_A$ has a direct summand whose dimension vector is not parallel to $(p,q)$.
\end{lemma}

\begin{proof}
Straightforward.
\end{proof}

\def\lcm{{\rm lcm}}

\begin{lemma}\label{l:aa}
Suppose $A\in\Mat_{p,q}^m$ and let $N_A$ be the corresponding $\Lambda_m$-module. Suppose that all the indecomposable direct summands in $N_A$ have dimension vectors parallel to $(p,q)$. Then $A$ and $\lambda A$ are in the same $\SL_p \times \SL_q$-orbit if and only if $\lambda$ is an $\lcm(p,q)$\textsuperscript{th} root of unity.
\end{lemma}

\begin{proof}
Let $p'=\frac{\lcm(p,q)}{q}$ and $q'=\frac{\lcm(p,q)}{p}$. 

$(\Leftarrow)$ Suppose $\lambda$ is an $\lcm(p,q)$\textsuperscript{th} root of unity.
If $a,b\in\Z$ are such that $ap'+bq'=1$, then $\lambda^{bq'}I_p\in \SL_p$, $\lambda^{ap'}I_q\in \SL_q$ and
$(\lambda^{bq'}I_p) A (\lambda^{ap'}I_q) =\lambda A$.

$(\Rightarrow)$ Suppose there exists $(P,Q) \in \SL_p \times \SL_q$ such that 
$PAQ = \lambda A$. 
Consider the linear map $L = L_{P,Q} : \Mat_{p,q} \to \Mat_{p,q}$ given by $L(X) = P X Q$. Since each $A_i$ is an eigenvector of $L$, it is also an eigenvector of $L^{\rm ss}$, the semisimple part of $L$ (from the Jordan--Chevalley decomposition). The map $L_{P,Q}$ is represented by the matrix $P \otimes Q^\tra$. 
Then $L^{\rm ss}$ is represented by $(P \otimes Q^\tra)^{\rm ss} = P^{\rm ss} \otimes (Q^{\tra})^{\rm ss}$, hence
$L^{\rm ss}=L_{P^{\rm ss}, Q^{\rm ss}}$. 
Since $P^{\rm ss}$ and $Q^{\rm ss}$ have the same determinant as $P$ and $Q$, 
we have $(P^{\rm ss},Q^{\rm ss}) \in \SL_p \times \SL_q$ and 
$P^{\rm ss}AQ^{\rm ss} = \lambda A$. Thus, without loss of generality, we can assume $P$ and $Q$ are semisimple.

We can then write $P = g D_1 g^{-1}$ and 
$Q= h D_2 h^{-1}$ for some $g\in\GL_p,h\in\GL_q$
and $D_1,D_2$ that are diagonal;
$D_1 = \alpha_1 I_{p_1}\oplus\cdots\oplus \alpha_k I_{p_k}$
with pairwise distinct $\alpha_i$
and 
$D_2 = \beta_1 I_{q_1}\oplus\cdots\oplus \beta_\ell I_{q_\ell}$
with pairwise distinct $\beta_j$.
Then $D_1 A'_i D_2 = \lambda A'_i$, where $A'_i = g^{-1} A_i h$. It is straightforward to see that the dimension vectors of the indecomposable summands of $A' = (A'_1,\dots, A'_m)$ are the same as for $A$ because $N_A \cong N_{A'}$.

Next we split each $A'_t$ into a $k \times \ell$ block matrix, where the $(i,j)$ block has size $p_i \times q_j$. Then left and right multiplication by $D_1$ and $D_2$ scales the $(i,j)$ block by $\alpha_i \beta_j$. So if this block is nonzero, we must have $\alpha_i \beta_j = \lambda$. 
Since the $\alpha_i$s are distinct and the $\beta_j$s are distinct, 
only one block in each block row and block column can be nonzero
(and this holds across all $A'_t$s simultaneously). 
In particular, each such block corresponds to a direct summand of $N_{A'}$, so our hypothesis on the dimension vectors of indecomposable summands implies $pq_j=qp_i$.
Moreover, an entire block column (resp. block row) cannot be zero because that yields a direct summand of dimension $(1,0)$ (resp. $(0,1)$), which contradicts the hypothesis. So we conclude that $k = \ell$ and, after a permutation of block rows,
$p_i = d_ip'$ and $q_i=d_iq'$ for some $d_i\in\N$. 

Then
\begin{align*}
1&=\det(D_1)=\prod_i \alpha_i^{p_i}=\Big(\prod_i \alpha_i^{d_i}\Big)^{p'},\\
1&=\det(D_2)=\prod_i \beta_i^{q_i}
=\Big(\prod_i \left(\frac{\lambda}{\alpha_i}\right)^{d_i}\Big)^{q'}
=\lambda^{q}\Big(\prod_i \alpha_i^{d_i}\Big)^{-q'},
\end{align*}
whence
\[\lambda^{\lcm(p,q)} = \lambda^{qp'}
=\Big(\prod_i \alpha_i^{d_i}\Big)^{p'q'}=1. \qedhere\]
\end{proof}

Specializing \cite[Theorem 8.1.3]{DW}
to the $m$-Kronecker quiver gives the following.

\begin{proposition} \label{prop:dim-indec-summand}
Let $A \in \Mat_{p,q}^m$, and let $N_A$ be the corresponding $\Lambda_m$-module. Then all the indecomposable direct summands of $N_A$ have dimension vectors parallel to $(p,q)$ if and only if there exists $C \in \Mat_{q,p}^m$ such that 
$\sum_{i=1}^m A_i C_i = qI_p$ and $\sum_{i=1}^m C_i A_i = pI_q$.
\end{proposition}

The $\SL_p\times\SL_q$ equivalence of $A$ and $\lambda A$ is thus summarized as follows.

\begin{corollary}\label{c:lam}
Let $A\in \Mat_{p,q}^m$ and $\lambda\in\C$. Then $A$ and $\lambda A$ lie in the same  $\SL_p\times\SL_q$-orbit if and only if one of the following conditions hold:
\begin{enumerate}[\rm(a)]
    \item $A=0$;
    \item $\lambda^{\lcm(p,q)}=1$;
    \item $\lambda\neq0$ and no $C\in\Mat_{q,p}^m$ satisfies
$\sum_{i=1}^m A_i C_i = qI_p$ and $\sum_{i=1}^m C_i A_i = pI_q$.
\end{enumerate}
\end{corollary}

\begin{proof}
The case $A=0$ is trivial. Thus we can assume that $A\neq 0$ and $\lambda\neq 0$. If $N_A$ admits an irreducible direct summand with dimension vector not parallel to $(p,q)$, then by Lemmas \ref{l:ab2}, \ref{l:ab} and Proposition \ref{prop:dim-indec-summand}, $A$ and $\lambda A$ are in the same orbit if and only if (c) holds. Otherwise, $A$ and $\lambda A$ are in the same orbit if and only if (b) holds by Lemma \ref{l:aa}.
\end{proof}

\section{Rank inequalities and orbit closure}\label{sec6}

In view of Proposition \ref{p:matpoly}, the Hadwin--Larson conjecture \cite[Conjecture 2]{HL} asks whether the following are equivalent for $A,B\in\Mat_n^m$:
\begin{enumerate}[\rm(a)]
	\item $A$ lies in the closure of the $\GL_n$-orbit of $B$;
	\item for all $N\in\N$ and $T=(T_0,\dots,T_m)\in \Mat_{N}^{m+1}$,
	$$\rk\left(I\otimes T_0 + A_1\otimes T_1+ \cdots A_m\otimes T_m\right) \le
	\rk\left(I\otimes T_0 + B_1\otimes T_1+ \cdots B_m\otimes T_m\right).$$
\end{enumerate}
Note that (a)$\Rightarrow$(b) is clear because the rank of a matrix is lower semi-continuous.
In this section we present an explicit counterexample to the Hadwin--Larson conjecture. In the language of degenerations of modules, it was given by Carlson \cite[Section 3.1]{Rie} (see also \cite[Section 7.2]{Bon1}) to distinguish between degenerations and virtual degenerations. Here we concretize it in terms of matrix tuples.

Let
$$
B_1=
\begin{pmatrix}
0&0&0&0\\
1&0&0&0\\
0&0&0&0\\
0&0&1&0
\end{pmatrix}, \quad
B_2=
\begin{pmatrix}
0&0&0&0\\
0&0&0&0\\
1&0&0&0\\
0&1&0&0
\end{pmatrix},\quad
A_1=A_2=B_1.
$$

First we claim that $(a)$ fails for $A$ and $B$. That is, $A$ is not in the closure of the $\GL_4$-orbit of $B$. Let $x_{ij},y_{ij}$ be the coordinates of the affine space $\Mat_4^2$, and
$$p=x_{43} y_{21}-x_{41} y_{23}-x_{23} y_{41}+x_{21} y_{43}.$$
A direct calculation shows that $p(PB_1P^{-1},PB_2P^{-1})=0$ for every $P\in\GL_4$, and $p(A_1,A_2)=2$.

On the other hand, $A\oplus 0_1$ lies in the closure of the $\GL_5$-orbit of $B\oplus 0_1$ (the argument in \cite[Section 3.1]{Rie} implies that $A\oplus 0_2$ lies in the $\GL_6$-orbit closure of $B\oplus 0_2$).
Indeed, for $t\neq0$ let
$$P_t = \begin{pmatrix}
1 & 0 & 0 & 0 & 0  \\
0 & 1 & 1 & 0 & 1  \\
0 & 0 & t^2 & 0 & 0  \\
0 & 0 & 0 & t^2 & 0  \\
0 & t & -t & 0 & 0 
\end{pmatrix} \in \GL_5.$$
Then
$$P_t(B\oplus 0_1)P_t^{-1}=
\left(
\begin{pmatrix}
0&0&0&0&0\\
1&0&0&0&0\\
0&0&0&0&0\\
0&0&1&0&0\\
t&0&0&0&0
\end{pmatrix},
\begin{pmatrix}
0&0&0&0&0\\
1&0&0&0&0\\
t^2&0&0&0&0\\
0&0&1&0&t\\
-t&0&0&0&0
\end{pmatrix}
\right)
$$
and so $A\oplus 0_1=\lim_{t\to0} P_t(B\oplus 0_1)P_t^{-1}$. Therefore (b) holds for $A\oplus 0_1$ and $B\oplus 0_1$, and consequently also for $A$ and $B$.

The above example indicates that (a) above should be replaced by
\begin{enumerate}
    \item[(a')] for some $\ell\in\N$, $A\oplus 0_\ell$ lies in the closure 
    of the $\GL_{n+\ell}$-orbit of $B\oplus 0_\ell$.
\end{enumerate}
The problem of equivalence of (a') and (b) has a counterpart in representation theory.
There it is the open question of whether the virtual degeneration order and the hom order are equivalent \cite[Section 5]{Sma} (more precisely, virtual degeneration allows for the zero tuple $0_\ell$ in (a') to be replaced by an arbitrary $C\in\Mat_\ell^m$).

\section{Algorithms}\label{sec7}

In this section we give algorithms pertaining to the main results of the paper.
A deterministic polynomial time algorithm for testing orbit equivalence under similarity by $\GL_n$ and the left-right action by $\GL_p \times \GL_q$ 
is a special case of the Brooksbank--Luks algorithm for testing isomorphism of finite-dimensional modules over a finitely generated algebra \cite[Theorem 3.5]{BL08} (incidentally, Algorithm \ref{a:wit} below also tests similarity, although this is not its chief purpose). There is also a very straightforward probabilistic procedure for testing similarity: given $A,B\in\Mat_n^m$, choose a random solution $P\in\Mat_n$ of the linear system $BP=PA$; then $A$ and $B$ are similar if and only if $P$ is invertible.

\subsection{Constructing a rank-disparity witness}\label{sec71}
Here we describe how, given a pair of non-similar tuples $A, B \in \Mat_{n}^m$, one can produce a tuple that witnesses the violation of the rank equality condition \eqref{e:rank}.
\begin{algorithm}\label{a:wit}
Construction of a rank-disparity witness.\\
\noindent{\bf Input:} $A, B \in \Mat_{n}^m$.
\begin{description}
\setlength\itemsep{.5em}

\item[Step 1] Construct the finite sequence of modules $L_{i+1}=\rad(\End L_i)\cdot L_i$. \\{\rm 
Determining the endomorphism ring of a finite-dimensional module over a finite-dimensional algebra amounts to solving a linear system, and likewise for determining the radical of a finite-dimensional algebra \cite[Corollary 4.3]{FR}.
}
	
\item[Step 2] Find indecomposable summands in each $L_i$. 
\\{\rm 
This is done by determining the ranges of centrally primitive idempotents of the semisimple part (as in the Wedderburn principal theorem) of the algebra $\End L_i$ \cite[Theorem 6]{CIK}.
}
	
\item[Step 3] For each indecomposable module $M_C$ from Step 2, construct a matrix tuple $T$ as in \eqref{e:c2t}. 
\\{\rm 
By Proposition \ref{p:fns} and the proof of Theorem \ref{t:main}, either one of them violates \eqref{e:rank} (in which case $A$ and $B$ are not similar), or $A$ and $B$ are similar.
}
\end{description}
\end{algorithm}

\begin{example}
In the counterexample \cite[Example 5]{HL} to the Curto--Herrero conjecture \cite[Conjecture 8.14]{CH}  it is shown that the pairs of $3\times 3$ matrices $A=(E_{12},E_{13})$ and $B=(E_{21},E_{31})$ satisfy $\rk f(A)=\rk f(B)$ for all $f\in \kk\Langle x_1,x_2\Rangle$, and
$\rk (I\otimes T_0+A_1\otimes T_1+A_2\otimes T_2)\neq \rk (I\otimes T_0+B_1\otimes T_1+B_2\otimes T_2)$ for
$$T_0=0_2,\quad
T_1=\begin{pmatrix}
1 & 0 \\ 0 & 0
\end{pmatrix},\quad
T_2=\begin{pmatrix}
0 & 0 \\ 1 & 0
\end{pmatrix}.$$
This concrete witness $T\in \Mat_2^3$ arises from 
the 1-dimensional module $M_{(0_1,0_1)}$
which is a direct summand in the 5-dimensional module $L_1$ as per Algorithm \ref{a:wit}. Both 3-dimensional indecomposable summands of $L_0$ (namely, $M_A$ and $M_B$) also give rank-disparity witnesses (in $\Mat_6^3$).
\end{example}

\subsection{Deciding orbit equivalence 
for the \texorpdfstring{$\SL_p \times \SL_q$}{SLpxSLq} action}\label{sec72}

We give two algorithms for testing $\SL_p \times \SL_q$ equivalence, one for points outside the null cone when $p=q$, and one for general points. Note that there is a deterministic polynomial time algorithm for the null cone membership \cite[Theorem 1.5]{IQS}.

Proposition \ref{p:notnull} leads to the following procedure.

\begin{algorithm} $\SL_n\times \SL_n$ equivalence outside the null cone.\\
\noindent{\bf Input:}  $A, B \in \Mat_{n}^m$, not in the null cone.
\begin{description}
\setlength\itemsep{.5em}
\item[Step 1] Check whether $A$ and $B$ are $\GL_n\times\GL_n$-equivalent by applying  \cite[Theorem 3.5]{BL08}. 
If they are not, then $A$ and $B$ are not $\SL_n\times \SL_n$-equivalent.
Otherwise, proceed to Step 2.

\item[Step 2] Using \cite[Theorem 1.5]{IQS}, find $T \in \Mat_{n-1}^m$ and $T' \in \Mat_n^m$ such that $\det(\sum_i A_i \otimes T_i) \neq 0$ and $\det(\sum_i A_i \otimes T'_i) \neq 0$.
    
\item[Step 3] $A$ and $B$ are $\SL_n \times \SL_n$-equivalent if and only if
$\det(\sum_i A_i \otimes T_i) = \det(\sum_i B_i \otimes T_i)$ and 
$\det(\sum_i A_i \otimes T'_i) = \det(\sum_i B_i \otimes T'_i)$.
\\
{\rm This holds by Proposition \ref{p:notnull}}.
\end{description}
\end{algorithm}

Finally, we give an algorithmic counterpart of Corollary \ref{c:lam}.

\begin{algorithm}\label{a:sl}
$\SL_p \times \SL_q$ equivalence in general.\\
\noindent{\bf Input:}  $A, B \in \Mat_{p,q}^m$.
\begin{description}
\setlength\itemsep{.5em}
\item[Step 1] Using \cite[Theorem 3.5]{BL08}, 
check whether $A$ and $B$ are $\GL_p\times\GL_q$-equivalent, and if so,
produce $(P,Q)\in \GL_p\times \GL_q$ such that $B=PAQ$.
If $A$ and $B$ are not $\GL_p\times \GL_q$-equivalent, 
then they are not $\SL_p\times \SL_q$-equivalent.
Otherwise, proceed to Step 2.

\item[Step 2] Check if the linear system $\sum_i A_iC_i=qI_p,\sum_i C_iA_i=pI_q$ in $C\in\Mat_{q,p}^m$ is consistent. If not, then $A$ and $B$ are $\SL_p \times \SL_q$-equivalent. Otherwise, proceed to Step 3.
\\
{\rm The validity of this step follows from Lemmas \ref{l:ab}, \ref{l:ab2}
and Proposition \ref{prop:dim-indec-summand}.
}

\item[Step 3] $A$ and $B$ are $\SL_p\times\SL_q$-equivalent if and only if 
$\det(P)\det(Q)=1$.
\\
{\rm 
Indeed, let $\lambda$ be an $\lcm(p,q)$\textsuperscript{th} root of $\det(P)\det(Q)$. Then $A$ and $B$ are $\SL_p\times \SL_q$-equivalent if and only if $A$ and $\lambda A$ are. By Lemma \ref{l:aa}, this is further equivalent to $\lambda^{\lcm(p,q)}=1$.
}
\end{description}
\end{algorithm}


\begin{thebibliography}{AA}

\bibitem[AKRS21]{AKRS21} C. Am\' endola, K. Kohn, P. Reichenbach, A. Seigal: {\it Invariant theory and scaling algorithms for maximum likelihood estimation}, SIAM J. Appl. Algebra Geom. 5 (2021) 304--337.

\bibitem[ACSY+]{ACSY}
O. Arizmendi, G. C\'ebron, R. Speicher, S. Yin:
\textit{Universality of free random variables: atoms for non-commutative rational functions}, preprint \url{https://arxiv.org/abs/2107.11507}

\bibitem[Aus82]{Aus82}
M. Auslander: \textit{Representation theory of finite-dimensional algebras}, in: Algebraists' homage: papers in ring theory and related topics (New Haven, Conn., 1981), 27--39, Contemp. Math. 13, Amer. Math. Soc.,  1982.

\bibitem[Bon96]{Bon1}
K. Bongartz:
{\it On degenerations and extensions of finite-dimensional modules}, 
Adv. Math. 121 (1996) 245--287.

\bibitem[BCR98]{BCR}
J. Bochnak, M. Coste, M. F. Roy:
{\it Real algebraic geometry},
Results in Mathematics and Related Areas 36, Springer, 1998.

\bibitem[BL08]{BL08}
P. A. Brooksbank, E. M. Luks:
{\it Testing isomorphism of modules},
J. Algebra 320 (2008) 4020--4029.

\bibitem[BFGO19]{BFGO+}
P. B\"urgisser, C. Franks, A. Garg, R. Oliveira, M. Walter, A. Wigderson:
{\it Towards a theory of non-commutative optimization: Geodesic 1st and 2nd order methods for moment maps and polytopes}, 
in 2019 IEEE 60th Annual Symposium on Foundations of Computer Science (FOCS), IEEE (2019) 845--861.

\bibitem[Coh06]{Coh}
P. M. Cohn:
{\it Free ideal rings and localization in general rings},
New Mathematical Monographs 3, Cambridge University Press, 2006.

\bibitem[CIK97]{CIK}
A. Chistov, G. Ivanyos, M. Karpinski:
{\it Polynomial time algorithms for modules over finite dimensional algebras},
Proceedings of the 1997 International Symposium on Symbolic and Algebraic Computation (Kihei, HI), 68--74, ACM, New York, 1997.

\bibitem[CH85]{CH}
R. E. Curto, D. A. Herrero:
{\it On closures of joint similarity orbits},
Integral Equations Operator Theory 8 (1985) 489--556.

\bibitem[CH02]{CH1}
J. Cui, J. Hou:
{\it Linear maps on von Neumann algebras preserving zero products or TR-rank},
Bull. Austral. Math. Soc. 65 (2002) 79--91. 

\bibitem[CH04]{CH2}
J. Cui, J. Hou:
{\it Completely rank nonincreasing linear maps on nest algebras},
Proc. Amer. Math. Soc. 132 (2004) 1419--1428. 

\bibitem[DKS04]{DKS}
K. R. Davidson, D. W. Kribs, M. E. Shpigel:
{\it Isometric dilations of non-commuting finite rank $n$-tuples},
Canad. J. Math. 53 (2001) 506--545. 

\bibitem[DM17]{DM17} 
H. Derksen, V. Makam: 
{\it Polynomial degree bounds for matrix semi-invariants}, Adv. Math. 310 (2017) 44--63.

\bibitem[DM20]{DM20}
H. Derksen, V. Makam:
{\it Algorithms for orbit closure separation for invariants and semi-invariants of matrices},
Algebra Number Theory 14 (2020) 2791--2813. 

\bibitem[DM21]{DM21}
H. Derksen, V. Makam: {\it Maximum likelihood estimation for matrix normal models via quiver representations},  SIAM J. Appl. Algebra Geom. 5 (2021) 338--365.

\bibitem[DW17]{DW}
H. Derksen, J. Weyman:
{\it An introduction to quiver representations},
Graduate Studies in Mathematics 184, American Mathematical Society, 2017.

\bibitem[Dro80]{Dro}
Ju. A. Drozd:
{\it Tame and wild matrix problems},
Representation theory II, 242--258, Lecture Notes in Math. 832, Springer, 1980.

\bibitem[EH88]{EH88} D. Eisenbud, J. Harris: 
{\it Vector spaces of matrices of low rank}, 
Adv. Math. 70 (1988) 135--155.

\bibitem[EH19]{EH}
E. Evert, J. W. Helton:
{\it Arveson extreme points span free spectrahedra},
Math. Ann. 375 (2019) 629--653.

\bibitem[FNS10]{FNS}
T. A. Forbregd, N. M. Nornes, S. O. Smal{\o}:
{\it Partial orders on representations of algebras}, 
J. Algebra 323 (2010) 2058--2062.

\bibitem[FR85]{FR}
K. Friedl, L. R\'onyai:
{\it Polynomial time solutions of some problems of computational algebra},
Proceedings of the Seventeenth Annual ACM Symposium on Theory of Computing (STOC 85), 153--162, ACM, 1985.

\bibitem[Fri83]{Fri}
S. Friedland:
{\it Simultaneous similarity of matrices}, 
Adv. Math. 50 (1983) 189--265.

\bibitem[GGOW16]{GGOW16} 
A. Garg, L. Gurvits, R. Oliveira, A. Wigderson: {\it A deterministic polynomial time algorithm for non-commutative rational identity testing}, 2016 IEEE 57th Annual Symposium on Foundations of Computer Science (FOCS). IEEE, 2016.

\bibitem[GGOW18]{GGOW-BL}
A. Garg, L. Gurvits, R. Oliveira, A. Wigderson: {\it Algorithmic and optimization aspects of Brascamp-Lieb inequalities, via operator scaling}, Geom. Funct. Anal. 28 (2018) 100--145.

\bibitem[HHY04]{HHY}
D. Hadwin, J. Hou, H. Yousefi:
{\it Completely rank-nonincreasing linear maps on spaces of operators},
Linear Algebra Appl. 383 (2004) 213--232. 

\bibitem[HL03]{HL}
D. Hadwin, D. R. Larson:
{\it Completely rank-nonincreasing linear maps},
J. Funct. Anal. 199 (2003) 210--227. 

\bibitem[HKV18]{HKV0}
J. W. Helton, I. Klep, J. Vol\v{c}i\v{c}:
{\it Geometry of free loci and factorization of noncommutative polynomials},
Adv. Math. 331 (2018) 589--626.

\bibitem[HKV22]{HKV}
J. W. Helton, I. Klep, J. Vol\v{c}i\v{c}:
{\it Factorization of noncommutative polynomials and Nullstellens{\"a}tze for the free algebra},
Int. Math. Res. Not. 1 (2022) 343--372.

\bibitem[HV07]{HV}
J. W. Helton, V. Vinnikov:
{\it Linear matrix inequality representation of sets},
Comm. Pure Appl. Math. 60 (2007) 654--674. 

\bibitem[HW14]{HW}
P. Hrube\v s, A. Wigderson:
{\it Non-commutative arithmetic circuits with division}, in Proceedings of the 5th conference on Innovations in theoretical computer science (2014) 49--66.

\bibitem[Hum97]{Hum}
J. E. Humphreys:
{\it Linear algebraic groups},
Graduate Texts in Mathematics 21, Springer, 1975.

\bibitem[IQS17]{IQS}
G. Ivanyos, Y. Qiao, K. V. Subrahmanyam:
{\it Constructive noncommutative rank computation in deterministic polynomial time over fields of arbitrary characteristics}, 
Comput. Complexity 27 (2018) 561--593.

\bibitem[KI04]{KI04}
V. Kabanets, R. Impagliazzo: 
{\it Derandomizing polynomial identity tests means proving circuit lower bounds}, Comput. complexity 13 (2004) 1--46.

\bibitem[KV17]{KV}
I. Klep, J. Vol\v{c}i\v{c}:
{\it Free loci of matrix pencils and domains of noncommutative rational functions},
Comment. Math. Helv. 92 (2017) 105--130.

\bibitem[Kip51]{Kip}
R. Kippenhahn:
{\it {\"U}ber den Wertevorrat einer Matrix},
Math. Nachr. 6 (1951) 193--228.

\bibitem[KLLR18]{Paulsen}
T. C. Kwok, L. C. Lau, Y. T. Lee, A. Ramachandran:
{\it The Paulsen problem, continuous operator scaling, and smoothed analysis}, in Proceedings of the 50th Annual ACM SIGACT Symposium on Theory of Computing (2018) 182--189.

\bibitem[Lam01]{Lam}
T. Y. Lam: 
{\it A First Course in Noncommutative Rings},
Graduate Texts in Mathematics 131, Springer, 2001.

\bibitem[LB97]{LB}
L. Le Bruyn:
{\it Orbits of matrix tuples},
Alg\`ebre non commutative, groupes quantiques et invariants (Reims, 1995), 245--261,
S\'emin. Congr. 2, Soc. Math. France, Paris, 1997. 

\bibitem[LBR99]{LBR}
L. Le Bruyn, Z. Reichstein:
{\it Smoothness in algebraic geography},
Proc. London Math. Soc. 79 (1999) 158--190. 

\bibitem[Mol99]{Mol}
L. Moln\'ar:
{\it Some linear preserver problems on $B(H)$ concerning rank and corank},
Linear Algebra Appl. 286 (1999) 311--321. 

\bibitem[Mul17]{Mul17}
K. Mulmuley:
{\it Geometric complexity theory V: Efficient algorithms for Noether normalization}, J. Amer. Math. Soc. 30 (2017) 225--309.

\bibitem[MS01]{MS01}
K. Mulmuley, M. Sohoni:
{\it Geometric complexity theory I: An approach to the P vs. NP and related problems}, SIAM J. Comput. 31 (2001) 496--526.

\bibitem[MFK94]{GIT} 
D. Mumford, J. Fogarty, F. Kirwan: \textit{Geometric invariant theory},
Ergebnisse der Mathematik und ihrer Grenzgebiete 34, Springer, 1994.

\bibitem[Pro76]{Pro}
C. Procesi:
{\it The invariant theory of {$n\times n$} matrices},
Adv. Math. 19 (1976) 306--381.

\bibitem[Rie86]{Rie} 
C. Riedtmann:
{\it Degenerations for representations of quivers with relations},
Ann. Sci. École Norm. Sup. (4) 19 (1986) 275--301. 

\bibitem[Sma08]{Sma}
S. O. Smal{\o}:
{\it Degenerations of representations of associative algebras},
Milan J. Math. 76 (2008) 135--164. 

\bibitem[Zwa00]{Zwa}
G. Zwara:
{\it Degenerations of finite-dimensional modules are given by extensions},
Compositio Math. 121 (2000) 205--218. 

\end{thebibliography}
\end{document}